\documentclass[12pt]{amsart}
\usepackage{amscd,verbatim}
\usepackage{amssymb,color}
\usepackage[colorlinks,linkcolor=blue,citecolor=blue,urlcolor=red]{hyperref}

\newcommand{\sP}{\mathcal{P}}

\newcommand{\F}{\mathbf{F}}

\newcommand{\Z}{\mathbf{Z}}
\newcommand{\uv}{{\underline{v}}}
\newcommand{\Or}{{\operatorname{Or}}}

\newcommand{\by}[1]{\overset{#1}{\longrightarrow}}
\newcommand{\iso}{\by{\sim}}
\newcommand{\osi}{\overset{\sim}{\longleftarrow}}
\renewcommand{\epsilon}{\varepsilon}
\newcommand{\St}{\operatorname{St}}
\newcommand{\AR}{\operatorname{AR}}
\newcommand{\ar}{\operatorname{ar}}
\newcommand{\rg}{\operatorname{rk}}

\newcommand{\Hom}{\operatorname{{Hom}}}
\newcommand{\Aut}{\operatorname{{Aut}}}
\newcommand{\Pic}{\operatorname{Pic}}
\newcommand{\Stab}{\operatorname{Stab}}
\newcommand{\Set}{\operatorname{\mathbf{Set}}}
\newcommand{\Spl}{\operatorname{\mathbf{Spl}}}
\newcommand{\Ord}{\operatorname{\mathbf{Ord}}}
\newcommand{\Top}{\operatorname{\mathbf{Top}}}
\newcommand{\Simpl}{\operatorname{Simpl}}
\newcommand{\Ker}{\operatorname{Ker}}
\newcommand{\Fr}{\operatorname{Fr}}
\newcommand{\GP}{\operatorname{GP}}
\newcommand{\sd}{\operatorname{sd}}
\newcommand{\Sq}{\operatorname{Sk}}
\newcommand{\sat}{{\operatorname{sat}}}

\newcommand{\inj}{\hookrightarrow}

\newcommand{\surj}{\rightarrow\!\!\!\!\!\rightarrow}

\renewcommand{\phi}{\varphi}
\renewcommand{\epsilon}{\varepsilon}

\newtheorem{thm}{Theorem}
\newtheorem{lemma}{Lemma}
\newtheorem{prop}{Proposition}
\newtheorem{cor}{Corollary}
\theoremstyle{remark}
\newtheorem{rk}{Remark}

\begin{document}
\title{On universal modular symbols}
\author{Bruno Kahn}
\address{IMJ-PRG\\Case 247\\4 place Jussieu\\75252 Paris Cedex 05\\France}
\email{bruno.kahn@imj-prg.fr}
\author{Fei Sun}
%\address{IMJ-PRG\\Case 247\\4 place Jussieu\\75252 Paris Cedex 05\\France}
\email{fei.sun@icloud.com}
\begin{abstract} We clarify the relationship between works of Lee-Szcza\-rba and Ash-Rudolph on the homology of the Steinberg module of a linear Tits building. This yields a simple proof of the Solomon-Tits theorem in this special case. We also give a (weak) relationship between this combinatorics and the one studied by van der Kallen, Suslin and Nesterenko to compute the homology of the general linear group with constant coefficients.
\end{abstract}
\date{Aug. 20, 2021}
\maketitle

\section*{Introduction}  In two related papers \cite{lee-sz,AR}, Lee-Szczarba and Ash-Rudolph study the homology of the Steinberg module of a Tits building by means of a canonical resolution \cite[Th. 3.1]{lee-sz} and an explicit set of generators called \emph{universal modular symbols} \cite[Prop. 2.3 and Th. 4.1]{AR}. A first purpose of this note is to clarify the relationship between the two approaches: we shall show in Theorem \ref{p3} that the generators provided by Lee and Szczarba coincide with the universal modular symbols of Ash and Rudolph: this answers a question asked in \cite[end of introduction]{AR}. 

For this, we offer in Theorem \ref{p1} a shorter proof of Lee-Szczarba's Theorem 3.1, which has the advantage to generalise from principal ideal domains to any integral domain $A$. As a byproduct, we get in Corollary  \ref{c1} a short proof of the Solomon-Tits theorem for $GL_n$.
We use the categorical techniques of Quillen \cite[\S 1]{quillen}.

%We also point out a small mistake in  of \cite[Prop. 2.3]{AR} for $n=2$. 

Finally, we give in Proposition \ref{p5} a (rather disappointing) relationship between the Lee-Szczarba resolution and the complexes used by van der Kallen, Suslin and Nesterenko to study the homology of the general linear group of an infinite field.

We thank Lo\"\i c Merel and Gregor Masbaum for helpful hints and  Jo\" el Riou for a critical reading of this note, especially for Remark \ref{rriou} in \S \ref{s2}.

\section{On the universal modular symbol for $n=2$} \label{s1}

Let us review the  Ash-Rudolph construction of the universal modular symbol in \cite[\S 2]{AR}. For coherence with the rest of this paper, we adopt a slightly different notation from theirs. For $r\ge 0$, let $\Delta_r$ be the standard (abstract) simplicial complex based on the set $[r]=\{0,\dots,r-1\}$:  the simplices of $\Delta_r$ are the nonempty subsets of $[r]$. Let $\sd \Delta_r$ denote the first barycentric subdivision of $\Delta_r$: the vertices of $\sd \Delta_r$ are the simplices of $\Delta_r$ and the simplices of $\sd \Delta_r$ are the nonempty sets of simplices of $\Delta_r$ which are totally ordered by inclusion (we shall call such a set a \emph{flag} of simplices). Its boundary $\partial\sd \Delta_r$ is the full subcomplex whose vertices are the nonempty proper subsets of $[r]$.

Let now $V$ be an $n$-dimensional vector space over a field $K$, with $n\ge 2$. The Tits building of $V$, denoted by $T(V)$, is the simplicial complex whose vertices are the (nonzero) proper subspaces of $V$ and simplices are flags of proper subspaces. It has the homotopy type of a bouquet of $(n-2)$-spheres by the Solomon-Tits theorem (\cite{solomon}, \cite[\S 2]{quillen2}; see also Corollary \ref{c1} below). Its $(n-2)$-th homology group is called the Steinberg module of $V$ and denoted by $\St (V)$.

Let $Q=(v_0,\dots,v_{n-1})$ be a sequence of $n$ nonzero vectors of $V$. It defines a simplicial map
\[\phi_Q:\partial \sd \Delta_{n-1}\to T(V)\]
by sending each vertex $I\subsetneq [n-1]$ to $\langle v_i\rangle_{i\in I}$. For $n>2$, the universal modular symbol $[v_0,\dots,v_{n-1}]\in \St (V)$ is defined as $(\phi_Q)_*\zeta$, where $\zeta\in H_{n-2}(\partial \sd \Delta_{n-1})$ is the fundamental class corresponding to the canonical orientation of $\sd \Delta_{n-1}$. By \cite[Prop. 2.2]{AR} the symbol $[Q]=[v_0,\dots,v_{n-1}]$  satisfies the following relations:
\begin{itemize}
\item[(a)] It is anti-symmetric (transposition of two vectors changes the sign of the symbol).
\item[(b)] It is homogeneous of degree zero: $[ a v_0 ,\dots, v_{n-1}] = [v_0,\dots,v_{n-1}]$ for any
nonzero $v_0,\dots,v_{n-1}$.
\item[(c)] $[Q] = 0$ if $\det Q = 0$.
\item[(d)] \label{rel}If $v_0,\dots,v_n$ are all non-zero, then
\[\sum_{i=0}^n(-1)^i[v_0,\dots,\hat{v_i},\dots,v_n]=0.\]
\item[(e)] If $A \in GL(V)$, then $[AQ] =A\cdot [Q]$, the dot denoting the natural action of
$GL(V)$ on $\St (V)$.
\end{itemize}

By \cite[Prop. 2.3]{AR}, the universal modular symbols generate $\St (V)$ (for $n>2$). Relations (a) -- (d) actually \emph{present} $\St (V)$ (Corollary \ref{c2}).

Let us look at the case $n=2$. Then $T(V)$ is the discrete set of lines of $V$, hence $\St (V)=H_0(T(V))$ is the free $\Z$-module over this basis.  The first problem is a definition of the ``fundamental class''  of the non connected discrete space $\partial \sd\Delta_1$. This space consists of the points $0,1$, which form a basis of $H_0(\partial \sd\Delta_1)$. If $\uv=(v_0,v_1)\in (V-\{0\})^2$, $\phi_\uv(i)$ is the line generated by $v_i$. The proof of Relation (c) above given on top of \emph{loc. cit.},  p. 244 is correct for $n>2$, but breaks down for $n=2$ since then  $H_0(\partial\Delta_1)\ne 0$. If we want to save this relation, we must make the right choice of the fundamental class: namely, $\zeta=[1]-[0]\in H_0(\partial \Delta_1)$. But then $[v_0,v_1]=[v_1]-[v_0]$, which is in the kernel of the augmentation $\St (V)\to \Z$ sending each line to $1$. Hence the symbols $[v_0,v_1]$ do not generate $\St (V)$, but rather the ``reduced Steinberg module"
$\widetilde{\St }(V)=\Ker(\St(V)\to \Z)$.

The above mistake is compensated by a  parallel error a little further: in \cite[Def. 3.1]{AR}, the second isomorphism does not exist for $n=2$ (the first author is indebted to Lo\"\i c Merel for pointing this out). The map goes the other way and yields an exact sequence
\[0=H_1(\bar X)\to H_1(\bar X,\partial\bar X)\to H_0(\partial\bar X)\to H_0(\bar X)=\Z\]
which gives an isomorphism $H_0(\partial\bar X)\iso \widetilde{\St}(V)$. This saves \cite[Prop. 3.2]{AR} for $n=2$. (In its proof, l. 4 one should read ``surjective" instead of ``injective".)

In the sequel, we shall write
\begin{equation}\label{eq5}
 \widetilde{\St}(V)=\begin{cases}
 \St(V)&\text{if $n>2$}\\
 \Ker(\St(V)\to \Z)&\text{if $n=2$}\\
 \Z &\text{if $n=1$}\\
 \Z &\text{if $n=0$.} 
 \end{cases}
\end{equation}

\section{Categories and functors}\label{s2}

We shall work with essentially $4$ categories: 
\begin{itemize}
\item $\Set$, the category of (small) sets.
\item $\Ord$, the category of partially ordered sets. Recall that, as in Quillen \cite{quillen}, we may think of a poset as a category.
\item $\Spl$, the category of abstract simplicial complexes \cite[3.1]{spanier}.
\item $\Top$, the category of topological spaces.
\end{itemize}

There are various functors between these categories: we write
\begin{itemize}
\item $E:\Set \to \Spl$  for the functor which sends a set $X$ to the simplicial complex of nonempty finite subsets of $X$.  %Exemple: $E([n]) = \Delta_n$ o\`u $[n]=\{0,\dots,n\}$. 
\item $B:\Ord\to \Spl$ for the functor sending a poset to the  simplicial complex of its  totally ordered nonempty finite subsets.
\item $\Simpl:\Spl\to \Ord$ for the functor which associates to a simplicial complex the set of its simplices ordered by inclusion. 
\item $|\ |:\Spl\to \Top$ for the geometric realisation functor \cite[3.1, pp. 110--111]{spanier}.
\end{itemize}

For any set $X$, we have $\Simpl E(X) = \sP_f(X)$, the poset of nonempty finite subsets of $X$. If $[n]\in \Set$ is the set $\{0,1,\dots,n-1\}$, then $E( [n]) = \Delta_n$, the standard $n$-simplex.

If $\omega:\Ord\to \Set$ is the forgetful functor, there is an obvious natural  transformation
\[\rho:B\Rightarrow E\circ \omega \]
and $\rho_S$ is an \emph{isomorphism} if $S$ is totally ordered, for example if $S=[n]$.

Note also that $B\circ \Simpl=\sd$ is the functor ``subdivision'' on simplicial complexes (remark of Segal to Quillen, \cite[p. 89]{quillen}).

Finally, we note the natural transfrmation
\[\theta:\Simpl \circ B\Rightarrow Id_{\Ord}\]
such that for $S\in \Ord$, $\theta_S$ maps $\sigma\in \Simpl B(S)$ to $\sup(\sigma)\in S$. Applying $B$ on the left, we get a natural transformation
\[B*\theta: \sd B\Rightarrow B.\]

Applying this to $S=[n]$, we get a canonical map
\begin{equation}
\sd \Delta_n \to \Delta_n
\end{equation}
which is natural for morphisms in $\Ord$ and induces a homotopy equivalence of (contractible) spaces after geometric realisation. From the definition of the latter, it extends to a homotopy equivalence
\begin{equation}\label{eq3}
\epsilon_\Gamma:|\sd \Gamma|\iso |\Gamma|
\end{equation}
which is natural in $\Gamma\in \Spl$.

For any $S\in \Ord$, $|B(S)|$ is naturally homeomorphic to $|N(S)|$, where $N(S)$ is the nerve of the category $S$; conversely, if $\Gamma\in \Spl$, the relation $B\circ \Simpl=\sd$ and \eqref{eq3} yield a natural  homotopy equivalence $|N(\Simpl(\Gamma))|\iso |\Gamma|$   (compare \cite[p. 89]{quillen}). Thus we can work equivalently with simplicial complexes or posets, and use Quillen's techniques from \cite{quillen} when dealing with the latter. Following the practice in \cite{quillen} and \cite{quillen2}, we shall say that a poset, a simplicial complex, or a morphism in $\Ord$ or $\Spl$ have a certain homotopical property if their topological realisations have.

\begin{rk}[J. Riou]\label{rriou} The morphism $\epsilon_\Gamma$ of \eqref{eq3} is not a homeomorphism in general, as the example $\Gamma=\Delta_1$ shows. On the other hand, the homeomorphism $|\sd \Gamma|\approx |\Gamma|$ constructed in \cite[3.3]{spanier} is not natural in $\Gamma$, as seen by considering the morphism $\Delta_2 \to \Delta_1$ identifying the vertices $1,2$. 
\end{rk}

The naturality of \eqref{eq3} is critical for the proof of Theorem \ref{p3} below.

\section{Some well-known lemmas}

\begin{lemma} \label{l1} $E(X)$ is contractible if $X$ is nonempty. 
\end{lemma}

\begin{proof} Here is one ``\`a la Quillen" (it is a version of the proof for simplicial sets):

Let $x\in X$ and let $\sP_f(X)_x$ be the subset of $\sP_f(X)$ consisting of those finite subsets that contain $x$. This poset has the smallest element $\{x\}$, hence is contractible. But the inclusion $\sP_f(X)_x\subset \sP_f(X)$ (viewed as a functor) has the left adjoint $Y\mapsto Y\cup \{x\}$.
\end{proof}

If $K\in \Spl$ and $r\ge 0$, we denote by $\Sq^r K$ its $r$-th skeleton: it has the same vertices as $K$ and its simplices are the simplices of $K$ of dimension $\le r$.

\begin{lemma}\label{l5} Let $\Gamma \in \Spl$, and let $v$ be a vertex of $\Gamma $. Then, for any $r\ge 0$, the map $\pi_i(\Sq^r \Gamma ,v)\to \pi_i(\Gamma ,v)$ is bijective for $i<r$ and surjective for $i=r$.
\end{lemma}

\begin{proof} An equivalent statement is: $\pi_i(\Gamma ,\Sq^r \Gamma )=0$ for $i\le r$. But the pair $(|\Sq^{r+1}\Gamma |,|\Sq^r\Gamma |)$ is $r$-connected by \cite[7.6, Lemma 15]{spanier}. By induction on $s$ this gives $\pi_i(\Sq^{r+s}\Gamma ,\Sq^r \Gamma )=0$ for $i\le r$ and any $s\ge 1$, hence the conclusion in the limit.
%But $|\Sq^{r+1}\Gamma |/|\Sq^r\Gamma |$ has by construction the homotopy type of a bouquet of $(r+1)$-spheres, where $|\ |$ denotes the geometric realisation. By induction on $s$ this gives $\pi_i(\Sq^{r+s}\Gamma ,\Sq^r \Gamma )=0$ for $i\le r$ and any $s\ge 1$, hence the conclusion in the limit.
\end{proof}

\begin{lemma}\label{l4} Let $X$ be a $r$-dimensional CW-complex which is $(r-1)$-connected. Then $X$ has the homotopy type of a bouquet of $r$-spheres.
\end{lemma}

Since we could not find a reference for this classical fact, here is a proof: si $r\le 1$, the statement is easy. If $r\ge 2$, the homology exact sequence
\[0=H_{r}(\Sq^{r-1} X)\to H_{r}(X)\to H_{r}(X,\Sq^{r-1} X)\]
injects $H_{r}(X)$ in the homology of a bouquet of $r$-spheres (see previous proof), showing that this group is \emph{free}\footnote{The first author thanks G. Masbaum for showing him this argument.}. Let $(e_i)_{i\in I}$ be a basis of $\pi_{r}(X,x)\iso H_{r}(X)$, where $x$ is some base-point, hence a map
\[f:\bigvee_{i\in I} S^r\to X\]
which is an isomorphism on $H_{r}$, hence a homology equivalence, hence a homotopy  equivalence (Whitehead's theorem, \cite[7.5, Th. 9]{spanier}).\qed

\begin{lemma}\label{l6} Let $\Gamma \in \Spl$. If $\Gamma $ is contractible, then $\Sq^r\Gamma $ has the homotopy type of a bouquet of $r$-spheres for any $r\ge 0$. Moreover, $H_r(\Sq^r \Gamma )$ is the $r$-th homology group of the (na\"\i vely) truncated complex $\sigma^{\le r} \Or_*(\Gamma )$, where $\Or_*(\Gamma )$ is the oriented chain complex of $\Gamma $ \cite[pp. 158--159]{spanier}.
\end{lemma}

\begin{proof} The first statement follows from Lemmas \ref{l5} and \ref{l4}. For the second one, we have $\Or_*(\Sq^r \Gamma )=\sigma^{\le r} \Or_*(\Gamma )$  tautologically.
\end{proof}

\section{A homotopy equivalence}\label{heq}

Let $A$ be an Noetherian domain with quotient field $K$, and let $M$ be a torsion-free finitely generated $A$-module. Write $V=K\otimes_R M$, so that $M$ is a lattice in $V$: we assume $\dim V=n\ge 2$. 

A submodule $N$ of $M$ is \emph{pure} if $M/N$ is torsion-free. Let $G^*(M)$ be the poset of proper pure submodules of $M$ (those different from $0$ and $M$). For $A=K$ we have $BG^*(V)=T(V)$ by definition, and by \cite[Prop. 4.2.4]{quillen-dedekind},  the map $N\mapsto K\otimes_R N$ yields a bijection
\[G^*(M)\iso G^*(V).\]

If $N\subset M$ is a submodule, the \emph{saturation} of $N$ is the smallest pure submodule $N_\sat$ of $M$ which contains $N$: it can be constructed as the kernel of the composition
\[M\to M/N\to (M/N)/\text{torsion}.\]

The following lemma is tautological:

\begin{lemma}\label{l2} Let $N\subseteq M$ be a pure submodule, and let $P$ be a submodule of $N$. Then $P_\sat\subseteq N$.\qed
\end{lemma}

The \emph{rank} of a subset $X$ of $M$ is the dimension of the subvector space of $V$ generated by $X$.  We write $E^*(M)$ for the set of nonempty finite subsets of rank $<n$ in $M-\{0\}$, viewed as a sub-simplicial complex of $E(M-\{0\})$. We then have a non-decreasing map:
\begin{align}\label{eq4}
\AR : \Simpl E^*(M)&\to G^*(M)\\
Y&\mapsto \langle Y\rangle_\sat.\notag
\end{align}

We take Quillen's viewpoint in \cite{quillen} and consider $\AR $ as a functor between the corresponding categories.

\begin{thm}\label{p1} $\AR $ is a homotopy equivalence. 
\end{thm}

\begin{proof} (Compare \cite[proof of Prop. 3.2]{lee-sz}.)  For $N\in G^*(M)$, we have by Lemma \ref{l2}
\[\AR / N = \sP_f(N-\{0\})\]
which is contractible (Lemma \ref{l1}). Apply \cite[Th. A]{quillen}.
\end{proof}

\begin{cor}[Solomon-Tits]\label{c1} $T(V)$ has the homotopy type of a bouquet of $(n-2)$-spheres.
\end{cor}

\begin{proof} We choose $A=K$ in Theorem \ref{p1}. On the one hand, the $p$-chains of $E^*(V)$ and $E(V-\{0\})$ coincide for $p\le n-2$, hence $T(V)$ is $(n-3)$-connected by Lemmas \ref{l1} and \ref{l5}.  On the other hand, $\dim T(V)\le n-2$. We conclude with Lemma \ref{l4}.
\end{proof} 

\section{The case of a principal ideal domain} \label{pid}

Keep the notation of the previous section. An element $v\in M$ is \emph{unimodular} if there exists a linear form $\theta:M\to A$ such that $\theta(v)=1$. We write $U(M)$ for the set of unimodular vectors of $M$.

\begin{lemma}\label{l3} If $A$ is principal, $U(M)\cap N$ is nonempty for any nonzero pure submodule $N\subseteq M$. 
\end{lemma}

\begin{proof} It suffices to prove this when $N$ has rank $1$. Then $N$ is free, with generator $v$. Since $M/N$ is torsion-free, it is free, hence $N$ is a direct summand in $M$. This readily implies that $v$ is unimodular.
\end{proof}

If $A$ is principal, let $U^*(M)$ be the set of nonempty finite subsets of rank $<n$ in $U(M)$: this is a sub-simplicial complex of $E^*(M)$.

\begin{prop}\label{p2}  The restriction $\AR^u$ of the functor $\AR$ of \eqref{eq4} to $\Simpl U^*(M)$ is a homotopy equivalence.
\end{prop}

\begin{proof} Same as for Theorem \ref{p1}, using Lemma \ref{l3}: here, $\AR^u/N=\sP_f(U(M)\cap N)$.
\end{proof}

\section{Comparison of the Ash-Rudolph and Lee-Szczarba constructions}  From \eqref{eq3} we get a zig-zag of isomorphisms  
\begin{equation}\label{eq2}
H_{n-2}(E^*(M))\osi H_{n-2}(\sd E^*(M))\iso \St(V)
\end{equation} 
induced by $B(\AR )$ and $\epsilon_{E^*(M)}$.

The singular chain complex of $E(M-\{0\})$ is given by
\[C_p(E(M-\{0\})) = \Z[(v_0,\dots,v_p)\mid v_i\in M-\{0\}].\]

That of $E^*(M)$ is given by
\[C_p(E^*(M)) = \Z[(v_0,\dots,v_p)\mid \rg \langle v_0,\dots, v_p\rangle <n].\]

Write $\bar C_*=C_*(E(M-\{0\}),E^*(M))$ for the quotient  complex. As $E(M-\{0\})$ is contractible, we have by Theorem \ref{p1}:
\[H_i(\bar C_*)\iso
\begin{cases}
 \widetilde{\St}(V)&\text{if $i=n-1$}\\
 0&\text{else}
 \end{cases}\]
(see \eqref{eq5} for $\widetilde{\St}(V)$).

Now $\bar C_p$ is isomorphic to the free $\Z$-module with basis the $(v_0,\dots, v_p)$ with $\dim \langle v_0,\dots v_p\rangle =n$. In particuliar, $\bar C_p=0$ for $p< n-1$.  Hence a resolution \`a la Lee-Szczarba \cite[th. 3.1]{lee-sz}:
\begin{equation}\label{eq1}
\dots\by{d_{n+1}}\bar C_n\by{d_n} \bar C_{n-1}\by{\ar} \widetilde{\St}(V)\to 0.
\end{equation}

To get back \cite[th. 3.1]{lee-sz} in the case where $A$ is principal (replacing $C_*(E^*(M))$ by $C_*(U^*(M))$), we use Proposition \ref{p2}.

\begin{thm}\label{p3} Modulo the isomorphisms of \eqref{eq2}, the map $\ar$ of \eqref{eq1} sends a generator $Q=(v_0,\dots,v_{n-1})$ to the universal modular symbol  $[v_0,\dots,v_{n-1}]$ of Ash-Rudolph.
\end{thm}

\begin{proof} The point is to get rid of subdivisions ``without calculation". For simplicity, write $\phi:=\phi_Q$.
Observe first  that $\phi$ factors as
\[\partial \sd\Delta_{n-1}\by{\tilde \phi} \sd E^*(M)= B\Simpl E^*(M)\by{B(\AR )} T(V)\]
where $\tilde \phi$ is the simplicial map sending a vertex $s$ of $\partial\sd \Delta_{n-1}$ to $\{v_i\mid i\in s\}$. 

There is an isomorphism of simplicial complexes  (induced by the inclusion $\partial\Delta_{n-1}\subset \Delta_{n-1}$)
\[\lambda:\sd \partial\Delta_{n-1} \iso \partial \sd\Delta_{n-1}.\]

The composition $\tilde \phi\circ \lambda$ is just $\sd \psi$, where $\psi: \partial\Delta_{n-1}\to E^*(V)$ is the restriction of $E(\Psi)$ with
\[\Psi:[n-1]\to M-\{0\}, \quad i\mapsto v_i.\]

By the  naturality of $\epsilon$ (\emph{cf.} \eqref{eq3}), we therefore have a commutative  diagram
\[\begin{CD}
|\sd\partial \Delta_{n-1}|@>|\lambda|>\sim>|\partial \sd\Delta_{n-1}|@>|\tilde \phi|>> |\sd E^*(V)|@>|B(\AR )|>> |T(V)|\\
@V{\epsilon_{\partial \Delta_{n-1}}}V{\wr}V &&@V{\epsilon_{E^*(V)}}V{\wr}V\\
|\partial \Delta_{n-1}| &@>|\psi|>> &|E^*(V)|.
\end{CD}\]

For $n>2$, if $\zeta'$ denotes the fundamental class  of $H_{n-1}(\sd \partial\Delta_{n-1})$ and $\zeta''$ denotes that of $H_{n-1}(\partial\Delta_{n-1})$, we have
\begin{align*}
\zeta &= \lambda_* \zeta'\\
\zeta'' &= (\epsilon_{\partial \Delta_{n-1}})_*\zeta'\\
[v_0,\dots,v_{n-1}]&=B(\AR )_*\circ \tilde\phi_*(\zeta)=B(\AR )_*\circ \tilde\phi_*\circ \lambda_*(\zeta').
\end{align*}

For $n=2$, define (\emph{cf.} \S \ref{s1}) the fundamental class $\zeta''$ of $H_{n-2}(\partial\Delta_{n-1})$ as the image of the ``positive"  generator of $H_{n-1}(\Delta_{n-1},\partial\Delta_{n-1})$, namely $[1] - [0]$, and $\zeta,\zeta'$ as the corresponding classes: the same identities hold.  It now suffices to show that $\psi_*(\zeta'')=(v_0,\dots,v_{n-1})\in H_{n-1}(E^*(V))$. 

For this, consider the commutative  diagram of exact sequences of complexes
\[\begin{CD}
0&\to& C_*(\partial\Delta_{n-1})@>>> C_*(\Delta_{n-1})@>>> C_*(\Delta_{n-1},\partial\Delta_{n-1})&\to& 0\\
&& @VVV @VVV @VVV\\
0&\to& C_*(E^*(V))@>>> C_*(E(V-\{0\})@>>> \bar C_*&\to& 0
\end{CD}\]
hence a commutative diagram
\[\begin{CD}
0&\to& H_{n-1}(\Delta_{n-1},\partial\Delta_{n-1})@>>> H_{n-2}(\partial\Delta_{n-1})@>>> H_{n-2}(\Delta_{n-1})\\
&& @VVV @VVV @VVV\\
0&\to& H_{n-1}(\bar C_*)@>>>H_{n-2}(E^*(V)) @>>> H_{n-2}(E(V-\{0\})
\end{CD}\]

For any $n\ge 2$, $\zeta''$ is the image of the element in $H_{n-1}(\Delta_{n-1},\partial\Delta_{n-1})$ represented by the cycle $z\in C_{n-1}(\Delta_{n-1},\partial\Delta_{n-1})$, image of the class  of the identity $\Delta_{n-1}\to \Delta_{n-1}$ in $C_{n-1}(\Delta_{n-1})$. By functoriality, the image of $z$ in $H_{n-1}(\bar C_*)$ is the image of $(v_0,\dots,v_{n-1})\in C_{n-1}(E(V-\{0\})$.
\end{proof}

\begin{cor} \label{c2} The group $\widetilde{\St}(V)$ is presented by the Ash-Rudolph relations (a)--(d) of \S \ref{s1}.
\end{cor}

\begin{proof}  Indeed, we may take $A=K$ in Theorem \ref{p3}; one should view $\bar C_{n-1}=C_{n-1}(E(V-\{0\})/C_{n-1}(E^*(V)))$ as the quotient of the free $\Z$-module with basis the $(v_0,\dots,v_{n-1})$ by the relations $(v_0,\dots,v_{n-1})\equiv 0$ if $\dim \langle v_0,\dots v_{n-1}\rangle <n$. This gives Relation (c), and Relation (d) comes from $d_n$. On the other hand, one easily checks that Relations (a) and (b) formally follow from (c) and (d). Namely,  by (c) and (d) we have the identity:
\begin{multline*}
\partial [g_0,\dots ,g_{i+1},g_i,g_{i+1}, \dots ,g_{n-1}]\\ 
=(-1)^i[g_0,\dots,g_{i+1},g_i, \dots, g_{n-1}] + (-1)^{i+2}[g_0,\dots, g_i, g_{i+1},\dots, g_{n-1}] = 0
\end{multline*}
which implies (a).  For (b), we have
\[\partial [g_0, ag_0, g_1, \dots , g_{n-1}] = [ag_0,g_1,\dots , g_{n-1}] -[g_0,g_1,\dots , g_{n-1}] = 0.\]
\end{proof}

\begin{rk} Together with Proposition, \ref{p2}, Theorem \ref{p3} also shows that Theorem 3.1 of \cite{lee-sz} implies Theorem 4.1 of \cite{AR}, \emph{cf.} \cite[end of introduction]{AR}.
\end{rk}

\section{The case of a Dedekind domain}

If $A$ is a Dedekind domain but is not principal, Lemma \ref{l3} is false even for $M=A^2$. Indeed, let $I\subset A$ be a nonprincipal ideal: it is generated by $2$ elements \cite[\S 1, Ex. 11 a) or \S 2, Ex. 1 a)]{bbki}, hence a surjection $A^2\surj I$ and an injection
\[I^*\inj A^2\]
where $I^*=\Hom(I,A)$. By construction $I^*$ is pure in $A^2$; if it contained a unimodular vector, there would be a linear form $\theta:A^2\to A$ such that $\theta_{|I^*}$ is surjective, hence bijective.  But then $I^*$, hence $I$, would be free, a contradiction. In fact:

\begin{lemma}\label{l3d} If $A$ is Dedekind, $U(M)\cap N\neq \emptyset$ for any pure submodule $N\subseteq M$ such that $\rg N>1$. If $\rg N=1$, $U(M)\cap N\neq \emptyset$ if and only if $N$ is free.
\end{lemma} 

\begin{proof} The case of rank $1$ is clear. In the other, recall that all torsion-free finitely generated $A$-modules are projective; by Steinitz's structure theorem for projective modules \cite[\S 4, no 10, Prop. 24]{bbki}, $N$ contains $A$ as  a direct summand. Since $N$ is itself a direct summand of $M$, it thus contains unimodular vectors.
\end{proof}

As in Proposition \ref{p2}, we thus get an equivalence
\begin{equation}\label{eq8}
\AR^u: \Simpl U^*(M) \iso G^{**}(M):=G^*(M) - G^1(M)
\end{equation}
where $G^1(M)=\{L\in G^*(M)\mid \rg L = 1 \text{ and }L\not \simeq A\}$. To compute further, we observe that the inclusion functor $T:G^1(M) \inj G^*(M)$ is \emph{cellular} in the sense of \cite[Def. 2.3.2]{quillen-dedekind}.\footnote{Recall that, by definition, this means that $T$ is fully faithful and $\Hom(d,c)=\emptyset$ for $d\in G^{**}(M)$ and $c\in G^1(M)$.} By \cite[Prop. 2.3.5]{quillen-dedekind}, we thus get a homotopy cocartesian square
\[\begin{CD}
G^{**}(M)\int \F_T @>p>> G^1(M)\\
@V{\epsilon}VV @V{T}VV\\
G^{**}(M) @>\iota >> G^*(M)
\end{CD}\]
where the category $G^{**}(M)\int \F_T$ has objects the inclusions $L\inj N$ for $L\in G^1(M)$, $N\in G^{**}(M)$, and morphisms the commutative squares. This category splits as a coproduct
\[G^{**}(M)\int \F_T=\coprod_{L\in G^1(M)} L\downarrow G^{**}(M). \]

The map $N\mapsto N/L$ induces an isomorphism of posets
\[L\downarrow G^{**}(M)\iso G^*(M/L). \]

Thus the square above becomes
\[\begin{CD}
\displaystyle\coprod_{L\in G^1(M)} G^*(M/L) @>p>> G^1(M)\\
@V{\epsilon}VV @V{T}VV\\
G^{**}(M) @>\iota >> G^*(M)
\end{CD}\]
where $p$ projects $G^*(M/L)$ onto $\{L\}$ and $\epsilon$  is the inverse image. Note that $ H_{n-2}(T(M/L))=0$ for all such $L$, and that $H_{n-3}(G^{**}(M))=0$ by \eqref{eq8} and by considering the chains of $C_*(E(M-\{0\}), E^*(M))$ as in the previous section. 
Hence an exact sequence
\[
 0 \to H_{n-2}(G^{**}(M))\to \tilde\St(M) \to \bigoplus_{L\in G^1(M)} \tilde\St(M/L) \to   0
\]
which gives a recursive computation of $\tilde\St(M)$ in terms of Ash-Rudolph symbols. In particular, taking homology, we find a long exact sequence
\begin{multline}\label{eq6.1}
\dots\to H_p(\Aut(M),H_{n-2}(G^{**}(M))\to H_p(\Aut(M),\tilde \St(M))\\
\to H_p(\Aut(M),\bigoplus_{L\in G^1(M)} \tilde\St(M/L))\to H_{p-1}(\Aut(M),H_{n-2}(G^{**}(M))\to \dots
\end{multline}

The group $\Aut(M)$ permutes the $L$'s, and permutes transitively those in a given isomorphism class (because $L$ is a direct summand of $M$). Hence in \eqref{eq6.1}, we have by Shapiro's lemma
\begin{multline*}
H_p(\Aut(M), \bigoplus_{L\in G^1(M)} \tilde\St(M/L)) \simeq \bigoplus_{\bar L\in \Pic(A)-\{ 0\}} H_p(\Stab_M(L),\tilde\St(M/L))
 \end{multline*}
where $\Stab_M(L)$ denotes the stabiliser of some $L\in \bar L$ in $M$ (note that its action on $\tilde\St(M/L)$ factors through the projection $\Stab_M(L)\to \Aut(M/L)$). For $p=0$, this boils down to $\bigoplus\limits_{\bar L\in \Pic(A)-\{ 0\}} \tilde\St(M/L)_{\Aut(M/L)}$.  This gives a recursive method to compute  $\tilde \St(M)_{\Aut(M)}$ in terms of unimodular symbols.

%{\color{red} 
%\begin{qn} Can one describe the differential $d^1$ in terms of the composite
%\begin{multline*}
%H_p(\Aut(M),\tilde \St(M))\to \bigoplus_{\bar L\in \Pic(A)-\{ 0\}} H_p(\Stab_M(L),\tilde\St(M/L))\\
%\to \bigoplus_{\bar L\in \Pic(A)-\{ 0\}} H_p(\Aut(M/L),\tilde\St(M/L))?
%\end{multline*}
%\end{qn}}

\section{Relatinship with the van der Kallen-Suslin-Nesterenko complexes} 

Let us now assume that $K$ is infinite. We say that  a (finite, nonempty) subset of $Y\subset M$ is a \emph{frame}  if the elements of $Y$ are linearly independent over $K$. We say that $Y$ is \emph{in general position} if any subset of $Y$ with at most $n$ elements is a frame. This defines two subcomplexes of $E(M-\{0\})$: 
\[\Fr(M)=\Sq^{n-1}\GP(M)\subset \GP(M)\subset E(M-\{0\}).\]

\begin{prop}\label{p4} $\GP(M)$ is contractible.
\end{prop}

\begin{proof} We adapt the proof of Lemma \ref{l1} in the style of \cite[Proof of Lemma 3.5]{NS}: for $v\in M$, let $\GP(M)_v=\{Y\in \GP(M)\mid v\in Y\}$ and $\GP(M)^v=\{Y\in \GP(M)\mid Y\cup \{v\}\in \GP(M)\}$. Since $\Simpl \GP(M)_v$ has a minimal element, it is contractible, hence so is $\GP(M)^v$ by the argument in the proof of Lemma \ref{l1}. Using that $K$ is infinite, for any $Y_1,\dots,Y_r\in \GP(M)$ there exists $v\in M-Y$ such that $Y_i\cup \{v\}\in \GP(M)$ for $i=1,\dots, r$. Hence any finite subcomplex $C$ of $\GP(M)$ is contained in some $\GP(M)^v$; thus the inclusion $C\to \GP(M)$ is nullhomotopic, hence the lemma.
\end{proof}

\begin{cor} $\Fr(M)$ has the homotopy type of a bouquet of $(n-1)$-spheres.
\end{cor}

\begin{proof} Apply Lemma \ref{l6}.
\end{proof}

There is an obvious inclusion $\Sq^{n-2} \GP(M)\subset E^*(M)$, hence a map
\begin{equation}\label{eq6}
H_{n-2}(\Sq^{n-2} \GP(M))\to H_{n-2}(E^*(M))\iso \St(V).
\end{equation}

\begin{prop}\label{p5} The map \eqref{eq6} is surjective.
\end{prop}

\begin{proof} Equivalently, we show that the map
\[H_{n-1}(\GP(M),\Sq^{n-2} \GP(M))\to H_{n-1}(E(M-\{0\}),E^*(M))\]
is surjective. Using Lemma \ref{l6}, these groups are obtained as the homology of the morphism of complexes 
\begin{equation}\label{eq7}
\Or_*(\GP(M))/\sigma^{\le n-2} \Or_*(\GP(M))\to \Or_*(E(M-\{0\})/\Or_*(E^*(M))
\end{equation}
(oriented chains). Both complexes are $0$ in degree $< n-1$, and \eqref{eq7} is an isomorphism in degree $n-1$.
\end{proof}

Unfortunately, \eqref{eq6} is far from being an isomorphism: for $n=2$ for example, its left hand side is free on the nonzero elements of $M$ while its right hand side is free on the lines of $M$ (or $V$). In particular, unlike its right hand side, the left hand side of \eqref{eq6} heavily depends on the choice of $A$ inside its field of fractions $K$. For a general $n$, the left hand side of \eqref{eq6} is presented by Relation (d) of p. \pageref{rel}.

\end{document}